\documentclass[12pt]{article}
\usepackage{amssymb}
\usepackage{amsmath}
\usepackage{amsfonts}
\usepackage{mitpress}
\usepackage{geometry}

\newtheorem{theorem}{Theorem}

\newtheorem{corollary}[theorem]{Corollary}

\newtheorem{remark}[theorem]{Remark}

\newenvironment{proof}[1][Proof]{\noindent\textbf{#1.} }{\ \rule{0.5em}{0.5em}}
\newdimen\dummy
\dummy=\oddsidemargin
\input{tcilatex}
\begin{document}

\begin{center}
\bigskip \textbf{Generalizations of four hyperbolic-type metrics and Gromov
hyperbolicity}

Marcelina Mocanu

\bigskip

Department of Mathematics and Informatics, \textquotedblleft Vasile
Alecsandri\textquotedblright\ University of Bac\u{a}u, Calea M\u{a}r\u{a}%
\c{s}e\c{s}ti 157, 600115 Bac\u{a}u, Romania; mmocanu@ub.ro

\bigskip
\end{center}

\textbf{Abstract}

We study in the setting of a metric space $\left( X,d\right) $ some
generalizations of four hyperbolic-type metrics defined on open sets $G$
with nonempty boundary in the $n-$dimensional Euclidean space, namely
Gehring-Osgood metric, Dovgoshey- Hariri-Vuorinen metric, Nikolov-Andreev
metric and Ibragimov metric. In the definitions of these generalizations,
the boundary $\partial G$ of $G$ and the distance from a point $x$ of $G$ to 
$\partial G$ are replaced by a nonempty proper closed subset $M$ of $X$ and
by a $1-$Lipschitz function positive on $X\setminus M$, respectively. For
each generalization $\rho $ of the hyperbolic-type metrics mentioned above
we prove that $\left( X\setminus M,\rho \right) $ is a Gromov hyperbolic
space and that the identity map between $\left( X\setminus M,d\right) $ and $%
\left( X\setminus M,\rho \right) $ is quasiconformal. For the Gehring-Osgood
metric and the Nikolov-Andreev metric we improve the Gromov constants known
from the literature. For Ibragimov metric the Gromov hyperbolicity is
obtained even if we replace the distance from a point $x$ to $\partial G$ by
any positive function on $X\setminus M$.



\section{Introduction}

More than three decades after it was introduced, the concept of Gromov
hyperbolicity is very relevant in metric geometry, due to its applications
to various fields of mathematics and theoretical computer science, bridging
discrete and continuous settings. In the seminal paper \cite{Gromov} Gromov
generalized the concept of negative curvature, as seen in the hyperbolic
plane, to a much broader class of metric spaces, in order to address
problems from geometric group theory and metric geometry. Gromov hyperbolic
spaces are used to study groups via their actions on spaces with hyperbolic
properties and have applications in topology, dynamical systems, and
computer science (e.g., in modeling network geometry).

In geodesic spaces, Gromov hyperbolic spaces are defined via the $\delta $%
-thin triangles condition, which characterizes how triangles in the space
resemble those in classical hyperbolic geometry. In general metric spaces,
Gromov hyperbolic spaces are defined using the Gromov product, as follows.
The Gromov product of $x,y\in X$ at $w\in X$ (with respect to a base point $%
w $) is $\left( \left. x\right\vert y\right) _{w}=\frac{1}{2}\left(
d(x,w)+d(y,w)-d\left( x,y\right) \right) $. The metric space $\left(
X,d\right) $ is said to be Gromov hyperbolic (with base point $w\in X$) if
there exists a constant $\delta \geq 0$ such that 
\begin{equation}
\left( \left. x\right\vert z\right) _{w}+\delta \geq \min \left\{ \left(
\left. x\right\vert y\right) _{w},\left( \left. y\right\vert z\right)
_{w}\right\}  \label{Gromovdef1}
\end{equation}%
for every $x,y,z\in X$ \cite{Hasto2006}. If this inequality holds for some $%
w\in G$ with a constant $\delta $, then it holds for every $w\in G$ with $%
\delta $ replaced by $2\delta $. The following symmetric form of (\ref%
{Gromovdef1}), with $\delta $ called a Gromov constant, will be used in the
sequel \cite{Hasto2006}: 
\begin{equation*}
d\left( x,z\right) +d\left( y,w\right) \geq \max \left\{ d\left( x,w\right)
+d\left( y,z\right) ,d\left( x,y\right) +d\left( z,w\right) \right\}
+2\delta .
\end{equation*}%
The article \cite{Vaisala} of V\"{a}is\"{a}l\"{a} is a mini monograph on
Gromov hyperbolic metric spaces, which need not be geodesic or proper. Lind%
\'{e}n \cite{Linden} gave in 2007 a status report on the research of various
hyperbolic-type metrics in proper subdomains $D$ of $\mathbb{R}^{n}$. In 
\cite{Linden} several examples of Gromov hyperbolic metrics $\rho _{D}$ on $%
D $ are collected and investigated, let us mention the classical hyperbolic
metric $h_{D}$ (where $D$ is a simply connected domain for $n=2$, while $%
D\in \left\{ \mathbb{B}^{n},\mathbb{H}^{n}\right\} $ for $n\geq 3$)- see 
\cite{BM}, the quasihyperbolic metric $k_{D}$ for an $A-$uniform domain $D$
(by \cite[Theorem 3.6]{Bonk}, $\delta \leq 10^{4}A^{8}$), the Apollonian
metric $\alpha _{D}$ (by \cite{Hasto2006}, using a rough isometry), the
Seitteranta metric $\delta _{D}$, the Ferrand metric $\sigma _{D}$ (if $D$
is a uniform domain, by \cite[Theorem 3.6]{Bonk}), Gehring-Osgood metric $%
\widetilde{j}_{D}$ (with $\delta \leq \log 3$, by \cite{Hasto2006}). The
half-Appolonian metric $\eta _{D}$ introduced by H\"{a}st\"{o} and Lind\'{e}%
n \cite{HaLi} and the distance-ratio metric $j_{D}$ introduced by Vuorinen 
\cite{Vuorinen} as a substitute for the Gehring-Osgood metric are very
useful in quasiconformal theory, but are Gromov hyperbolic (both with $%
\delta \leq \log 9$) only when $D$ has a single boundary point.

Many other hyperbolic-type metrics than that mentioned above have been
introduced and studied, such as the triangular ratio metric \cite{Hasto1},
the Cassinian metric \cite{IbraCas}, the visual angle metric \cite{Klen},
the Dovgoshey-Hariri-Vuorinen metric \cite{DHV2016}, the Nikolov-Andreev
metric \cite{Nikolov}, the Ibragimov metric \cite{Ibra}. For a systematic
account on these metrics and their applications to geometric function
theory, see the monograph by Hariri, Kl\'{e}n and Vuorinen \cite{HaririBook}%
. See also \cite{Deza}, where hyperbolic-type metrics are discussed among
may other metrics.

Very recently, the Gromov hyperbolicity in any proper subdomain of $\mathbb{R%
}^{n}$ was proved by Zhou, Zheng, Ponnusamy and Guan \cite{Zhou} for the
Dovgoshey-Hariri-Vuorinen metric $h_{c}$ (with $\delta =\log \frac{2c+1}{c}$%
), and by Luo, Rasila, Wang and Zhuo \cite{Luo} for the Nikolov-Andreev
metric (with $\delta =\log 15$).

In this paper we study generalizations of several hyperbolic-type metrics
defined on open sets $G$ with nonempty boundary in the $n-$dimensional
Euclidean space $\mathbb{R}^{n}$, namely Gehring-Osgood metric, Dovgoshey-
Hariri-Vuorinen metric, Nikolov-Andreev metric and Ibragimov metric. We
follow three directions of generalization. Firstly, the $n-$dimensional
Euclidean space is replaced by an arbitrary metric space $\left( X,d\right) $%
, which is natural, as Gromov hyperbolicity is defined in the general
setting of metric spaces \cite{Gromov}, \cite{Burago}, see also the papers
of V\"{a}is\"{a}l\"{a} \cite{Vaisala}, Ibragimov \cite{Ibra}, Zhang and
Xiao\ \cite{Zhang}. Secondly, the boundary of $G$ is replaced by a nonempty
proper closed subset $M$ of $X$, as in \cite{Ibra}. Thirdly, the distance
from a point $x$ to the boundary of $G$ is replaced by a positive $1-$%
Lipschitz function on $\left( X\setminus M,d\right) $, as in \cite[%
Proposition 11]{Nikolov}. For each generalization $\rho $ of the mentioned
hyperbolic-type metrics we prove that $\left( X\setminus M,\rho \right) $ is
a Gromov hyperbolic space and that the identity map between $\left(
X\setminus M,d\right) $ and $\left( X\setminus M,\rho \right) $ is
quasiconformal. or the Gehring-Osgood metric and the Nikolov-Andreev metric
we improve the Gromov constants known from the literature. For Ibragimov
metric the Gromov hyperbolicity is obtained even if we replace the distance
from a point $x$ to $\partial G$ by any positive function on $X\setminus M$.

Let $\left( X,d\right) $ be a metric space. The distance from a point $x\in
X $ to a non-empty subset $A$ of a metric space $\left( X,d\right) $ is $%
d_{A}\left( x\right) =\mathrm{dist}\left( x,A\right) =\inf \left\{
d(x,y):y\in A\right\} $ and is always a non-negative number. Moreover, $%
d_{A}(x)=0$ if and only if $x\in \overline{A}$. It is well-known that for
every nonempty set $A\subset X$ the function $d_{A}$ is $1-$Lipschitz on $X$%
. In particular, let $G\subset X$ be an open set with non-empty boundary.
Since $G$ is open, $G\cap \partial G$ is empty, hence $d_{\partial G}(x)>0$
for every $x\in G$. Recall that in a connected topological space every
proper non-empty open subset $G$ has a non-empty boundary, since $\emptyset
\neq \overline{G}\setminus G\subset \partial G$. The definitions of the
distance $\rho \left( x,y\right) $ between $x,y\in G$, in the sense of
Gehring-Osgood, Vuorinen, Dovgoshey- Hariri-Vuorinen and Nikolov-Andreev all
involve $d\left( x,y\right) $, $d_{\partial G}(x)$ and $d_{\partial G}(y)$.

We will show that the four generalized hyperbolic-type metrics do not change
the quasiconformal geometry of the space, using the below metric definition
of quasiconformal maps between metric spaces \cite{HeKo}. Given a
homeomorphism $f$ from a metric space $\left( X,d\right) $ to a metric space 
$\left( Y,d^{\prime }\right) $, then for $x\in X$ and $r>0$ we set 
\begin{equation*}
H_{f}\left( x,r\right) =\frac{\sup \left\{ d^{\prime }(f(x),f(y)):d\left(
x,y\right) \leq r\right\} }{\inf \left\{ d^{\prime }(f(x),f(y)):d\left(
x,y\right) \geq r\right\} }.
\end{equation*}%
A homeomorphism $f:\left( X,d\right) \rightarrow \left( Y,d^{\prime }\right) 
$ is called $H-$quasiconformal, with a nonnegative constant $H<\infty $, if $%
\underset{r\rightarrow 0}{\lim \sup }H_{f}\left( x,r\right) \leq H$ for
every $x\in X$. Intuitively, a $H-$quasiconformal homeomorphism between open
sets in $\mathbb{R}^{n}$ maps infinitesimal balls to infinitesimal
ellipsoids with a bounded ratio between the maximum axis and the minimum
axis.

\section{A generalization of the Gehring-Osgood metric}

Let $\left( X,d\right) $ be a metric space and let $G\subset X$ be an open
set with non-empty boundary. For every $x\in X$ consider $\delta \left(
x\right) =d_{\partial G}\left( x\right) =\mathrm{dist}\left( x,\partial
G\right) $, the distance from $x$ to the boundary of $G$.

The Gehring-Osgood metric in $G$ is defined by 
\begin{equation*}
\widetilde{j}_{G}\left( x,y\right) =\frac{1}{2}\log \left( 1+\frac{d(x,y)}{%
\delta \left( x\right) }\right) \left( 1+\frac{d(x,y)}{\delta \left(
y\right) }\right) ,
\end{equation*}%
for all $x,y\in G$ \cite{GehringOs}.

The distance ratio metric was introduced by Vuorinen \cite{Vuorinen} and is
defined by $j_{G}\left( x,y\right) =\log \left( 1+\frac{d(x,y)}{\min \left\{
\delta \left( x\right) ,\delta \left( y\right) \right\} }\right) $.

$\widetilde{j}_{G}$ and $j_{G}$ are well-defined, as $G\cap \partial G$ is
empty. Recall that in a connected topological space every proper non-empty
open subset $G$ has a non-empty boundary, since $\emptyset \neq \overline{G}%
\setminus G\subset \partial G$.

The metrics $\widetilde{j}_{G}$ and $j_{G}$ have been used in the study of
quasiconformal maps, in the setting of $n$-dimensional Euclidean spaces $X=%
\mathbb{R}^{n}$ \cite{GehringOs}, \cite{Vuorinen}, \cite{VuorinenBook}.

H\"{a}st\"{o} \cite{Hasto2006} proved that the Gehring-Osgood metric in an
open set $G\subsetneq \mathbb{R}^{n}$ is Gromov hyperbolic, with a constant
satisfying $\delta \leq \log 3$, where $\log $ is the natural logarithm. The
proof can be extended without modifications to the case where $G$ is an open
set with non-empty boundary in a metric space.

The aim of this section is to obtain a better estimate for the Gromov
hyperbolicity constant $\delta $ of the Gehring-Osgood metric, following
steps similar to that from H\"{a}st\"{o}'s proof, but giving a simpler
structure to the proof.

It is known that $\widetilde{j}_{G}$ is a metric when $X=\mathbb{R}^{n}$.
The case where $\left( X,d\right) $ is a general metric space was studied by
Ibragimov (\cite[3.1]{Ibra}). We will give a proof for a more general case,
where the distance to a closed nonempty proper subset $M$ of $X$ is replaced
by a $1-$Lipschitz function positive on $X\setminus M$.

\begin{theorem}
\label{GOmetric_Gen}Let $\left( X,d\right) $ be a metric space and $M$ be a
closed nonempty proper subset of $X$. Let $F:\left( X\setminus M,d\right)
\rightarrow \left( 0,\infty \right) $ be a $1-$Lipschitz function. The
function defined by $j\left( x,y\right) =\frac{1}{2}\log \left( 1+\frac{%
d(x,y)}{F\left( x\right) }\right) \left( 1+\frac{d(x,y)}{F\left( y\right) }%
\right) $ for $x,y\in X\setminus M$ is a metric on $X\setminus M$.
\end{theorem}

\begin{proof}
It suffices to prove the triangle inequality $j\left( x,y\right) \leq
j\left( x,z\right) +j\left( z,y\right) $ for every $x,y,z\in X\setminus M$. 
\newline
Denote $d\left( y,z\right) =a$, $d\left( z,x\right) =b$ and $d\left(
x,y\right) =c$. Also, denote $\frac{1}{F(x)}=X$, $\frac{1}{F\left( y\right) }%
=Y$ and $\frac{1}{F\left( z\right) }=Z$. \newline
The $1-$Lipschitz property of $F$ shows that $\left\vert \frac{1}{X}-\frac{1%
}{Y}\right\vert \leq c$, $\left\vert \frac{1}{Y}-\frac{1}{Z}\right\vert \leq
a$ and $\left\vert \frac{1}{Z}-\frac{1}{X}\right\vert \leq b$. The above
triangle inequality for the function $j$ is equivalent to 
\begin{equation}
\left( 1+cX\right) \left( 1+cY\right) \leq \left( 1+bX\right) \left(
1+bZ\right) \left( 1+aY\right) \left( 1+aZ\right) .  \label{To_Do_GO}
\end{equation}%
We have $c\leq a+b$, by triangle inequality for the metric $d$ on $X$. Then
the left-hand side in (\ref{To_Do_GO}) can be estimated as $\left(
1+cX\right) \left( 1+cY\right) \leq \left( 1+\left( a+b\right) X\right)
\left( 1+\left( a+b)\right) Y\right) $, i.e. 
\begin{equation}
\left( 1+cX\right) \left( 1+cY\right) \leq \left( 1+aY\right) \left(
1+bX\right) +aX\left( 1+aY\right) +bY\left( 1+bX\right) +abXY.  \label{Step1}
\end{equation}%
On the other hand, the right-hand side in (\ref{To_Do_GO}) can be written as 
\begin{equation*}
\left( 1+aY\right) \left( 1+bX\right) \left( 1+aZ+bZ+abZ^{2}\right) .
\end{equation*}%
But $\frac{1}{Z}-\frac{1}{Y}\leq a$, whence $Z\geq \frac{Y}{1+aY}$,
therefore \newline
$\left( 1+aY\right) \left( 1+bX\right) bZ\geq bY\left( 1+bX\right) $. 
\newline
Similarly, $\frac{1}{Z}-\frac{1}{X}\leq b$ implies $Z\geq \frac{X}{1+bX}$,
hence $\left( 1+aY\right) \left( 1+bX\right) aZ\geq aX\left( 1+aY\right) $. 
\newline
Moreover, multiplying the above estimates of $Z$ we get $\left( 1+aY\right)
\left( 1+bX\right) abZ^{2}\geq abXY$. \newline
Adding the latter three inequalities we get 
\begin{eqnarray}
&&\left( 1+bX\right) \left( 1+bZ\right) \left( 1+aY\right) \left( 1+aZ\right)
\label{Step2} \\
&\geq &\left( 1+aY\right) \left( 1+bX\right) +aX\left( 1+aY\right) +bY\left(
1+bX\right) +abXY.  \notag
\end{eqnarray}%
By (\ref{Step1}) and (\ref{Step2}) we get (\ref{To_Do_GO}).
\end{proof}

\begin{corollary}
For every metric space $\left( X,d\right) $ and every open set $G\subset X$
with non-empty boundary, the function defined by $\widetilde{j}_{G}\left(
x,y\right) =\frac{1}{2}\log \left( 1+\frac{d(x,y)}{\delta \left( x\right) }%
\right) \left( 1+\frac{d(x,y)}{\delta \left( y\right) }\right) $, $x,y\in
X\setminus \partial G$ is a metric on $G\cup \left( X\setminus \overline{G}%
\right) $.
\end{corollary}

\begin{proof}
The function $\delta \left( \cdot \right) =\mathrm{dist}\left( \cdot
,\partial G\right) $ is $1-$Lipschitz on $X$ and $\delta \left( x\right) =0$
if and only if $x\in \partial G$. Since $G\cap \partial G$ is empty, $\delta
(x)\neq 0$ for every $x\in G$. We apply Theorem \ref{GOmetric_Gen} with $%
M=\partial G$ and $F\left( x\right) =\delta \left( x\right) $, $x\in
X\setminus \partial G$.
\end{proof}

\begin{remark}
From the proof of Theorem \ref{GOmetric_Gen} we can see that $\widetilde{j}%
_{G}\left( x,y\right) =\widetilde{j}_{G}\left( x,z\right) +\widetilde{j}%
_{G}\left( z,y\right) $ if and only if the following three equalities hold: $%
d\left( x,y\right) =d(x,z)+d(z,y)$, $\mathrm{dist}\left( z,\partial G\right)
=\mathrm{dist}\left( x,\partial G\right) +d\left( x,z\right) $ and $\mathrm{%
dist}\left( z,\partial G\right) =\mathrm{dist}\left( y,\partial G\right)
+d\left( y,z\right) $. For example, for $x,y,z\in G=\mathbb{H}^{n}\subset 
\mathbb{R}^{n}$ we get $\widetilde{j}_{G}\left( x,y\right) =\widetilde{j}%
_{G}\left( x,z\right) +\widetilde{j}_{G}\left( z,y\right) $ if there exist $%
h\in \mathbb{R}^{n-1}$ and $0<y_{n}<z_{n}<x_{n}$ in $\mathbb{R}$ such that $%
x=\left( h,x_{n}\right) $, $y=\left( h,y_{n}\right) $ and $z=\left(
h,z_{n}\right) $.
\end{remark}

\begin{theorem}
\label{G-O}Let $\left( X,d\right) $ be a metric space and $M$ be a closed
nonempty proper subset of $X$. Let $F:\left( X\setminus M,d\right)
\rightarrow \left( 0,\infty \right) $ be a $1-$Lipschitz function. The
metric defined by $j\left( x,y\right) =\frac{1}{2}\log \left( 1+\frac{d(x,y)%
}{F\left( x\right) }\right) \left( 1+\frac{d(x,y)}{F\left( y\right) }\right) 
$ for $x,y\in X\setminus M$ is Gromov hyperbolic with a constant $\delta
\leq \frac{1}{4}\log 24$.
\end{theorem}

\begin{proof}
Let $x,y,z,w\in X\setminus M$. Denote $d\left( x,w\right) =a$, $d\left(
x,y\right) =b$, $d\left( y,z\right) =c$, $d\left( z,w\right) =d$, $d\left(
y,w\right) =e$ and $d\left( x,z\right) =f$. Also, we denote $\frac{1}{%
F\left( x\right) }=X$, $\frac{1}{F\left( y\right) }=Y$, $\frac{1}{F\left(
z\right) }=Z$ and $\frac{1}{F\left( w\right) }=W$. \newline
\emph{Step 1.} We look for $\delta >0$ independent on $x,y,z,w$ such that 
\begin{equation*}
j(x,z)+j(y,w)\leq \max \left\{ j(x,w)+j(y,z),j(x,y)+j(z,w)\right\} +2\delta .
\end{equation*}%
Denoting $e^{4\delta }=k$ and using the above notations, the previous
inequality is equivalent to the following logic disjunction: 
\begin{eqnarray*}
(1+fX)\left( 1+fZ\right) (1+eY)(1+eW) &\leq &k(1+aX)(1+aW)(1+cY)\left(
1+cW\right) \text{ or} \\
(1+fX)\left( 1+fZ\right) (1+eY)(1+eW) &\leq &k(1+bX)(1+bY)(1+dZ)\left(
1+dW\right) .
\end{eqnarray*}%
\emph{Step 2.} As in the proof of H\"{a}st\"{o}, we may assume that $e\leq f$%
, by symmetry. \newline
\emph{Step 3.} We also may assume that $\min \left\{ a,b,c,d\right\} =a$. 
\newline
This condition will be satisfied after using the following changes: $S1$
(swap $x$ and $z$) and/ or $S2$ (swap $y$ and $w$). We use $S2$ if $\min
\left\{ a,b,c,d\right\} =b$, $S1$ and $S2$ if $\min \left\{ a,b,c,d\right\}
=c$, respectively $S1$ if $\min \left\{ a,b,c,d\right\} =d$. \newline
\emph{Step 4.} By triangle inequality, $e\leq \min \left\{ a+b,c+d\right\} $
and $f\leq \min \left\{ a+d,b+c\right\} $. \newline
\emph{Step 5.} Since we assumed $\min \left\{ a,b,c,d\right\} =a$, it
follows that $e\leq 2b$ and $f\leq 2d$. \newline
\emph{Step 6.} $e\leq f$ and $f\leq 2d$ \ imply $e\leq 2d$. \newline
\emph{Step 7.} The $1-$Lipschitz continuity of $F$ implies $\left\vert \frac{%
1}{X}-\frac{1}{Y}\right\vert \leq b$, $\left\vert \frac{1}{X}-\frac{1}{Z}%
\right\vert \leq f$, $\left\vert \frac{1}{X}-\frac{1}{W}\right\vert \leq a$, 
$\left\vert \frac{1}{Y}-\frac{1}{Z}\right\vert \leq c$, $\left\vert \frac{1}{%
Y}-\frac{1}{W}\right\vert \leq e$ and $\left\vert \frac{1}{Z}-\frac{1}{W}%
\right\vert \leq d$.\newline
\emph{Step 8.} The function $t\mapsto \frac{1+pt}{1+qt}$, $t\in \left[
0,\infty \right) $ is increasing for $p>q$ and decreasing for $p<q$, hence $%
\frac{1+pt}{1+qt}\leq \max \left\{ 1,\frac{p}{q}\right\} $ for every $t\in %
\left[ 0,\infty \right) $ and this inequality is sharp. \newline
\emph{Step 9.} By Step 8 and Step 5, $\frac{1+fZ}{1+dZ}\leq 2$ and $\frac{%
1+eY}{1+bY}\leq 2$. \newline
Similarly, by Step 8 and Step 6, $\frac{1+eW}{1+dW}\leq 2$. It remains to
estimate from above $\frac{1+fX}{1+bX}$. \newline
\emph{Step 10. }In the following we will compare $d$\ and $e$. \newline
\textbf{Case 1. }Assume $d\leq e$. \newline
Then $f\leq a+d\leq a+e$, but $a+e\leq 3b$, by Step 3 and Step 5. \newline
Using Step 8, it follows that $\frac{1+fX}{1+bX}\leq 3$. This inequality and
inequalities from Step 9 imply that the second inequality from Step 1 is
satisfied with $k=24$. \newline
\textbf{Case 2. }Assume $e<d$. We will use an idea from the proof of H\"{a}st%
\"{o}. The function $t\mapsto \frac{1+et}{1+dt}$ is decreasing. From Step 7
we get $\frac{1}{W}\leq a+\frac{1}{X}$, hence $W\geq \frac{X}{1+aX}$. Then $%
\frac{1+eW}{1+dW}\leq \frac{1+e\frac{X}{1+aX}}{1+d\frac{X}{1+aX}}=\frac{%
1+(a+e)X}{1+\left( a+d\right) X}$, hence $\frac{1+fX}{1+bX}\cdot \frac{1+eW}{%
1+dW}\leq \frac{1+fX}{1+\left( a+d\right) X}\cdot \frac{1+\left( a+e\right) X%
}{1+bX}$. \newline
By Step 4 and Step 8, we get $\frac{1+fX}{1+\left( a+d\right) X}\leq 1$. 
\newline
By Step 3 and Step 5, $a+e\leq 3b$, hence using Step 8, $\frac{1+\left(
a+e\right) X}{1+bX}\leq 3$. \newline
The latter two inequalities and inequalities from Step 9 imply that the
second inequality from Step 1 is satisfied with $k=12$ in Case 2. \newline
\emph{Step 11. }The claim follows with $k=24$, therefore taking logarithms
and dividing by $2$ we obtain the Gromov hyperbolicity of the metric $j$
with a constant $\delta \leq \frac{1}{4}\log 24$. Note that $\log 3-\frac{1}{%
4}\log 24>0.3$.
\end{proof}

\begin{corollary}
For every metric space $\left( X,d\right) $ and every $G\subset X$ with
non-empty boundary, the Gehring-Osgood metric $\widetilde{j}_{G}$ is Gromov
hyperbolic on $G$ with a Gromov constant $\delta \leq \frac{1}{4}\log 24$.
\end{corollary}

\begin{theorem}
Let $\left( X,d\right) $ be a metric space, $M$ be a nonempty proper closed
subset of $X$ and $j$ be the generalized Gehring-Osgood metric from Theorem %
\ref{GOmetric_Gen}. Then the identity map $1_{X\setminus M}:\left(
X\setminus M,d\right) \rightarrow \left( X\setminus M,j\right) $ is $1-$%
quasiconformal.
\end{theorem}

\begin{proof}
Let $x\in X$. For every $y\in X$ such that $d(x,y)<F(x)$, 
\begin{equation}
j\left( x,y\right) \leq \frac{1}{2}\log \left( 1+\frac{d(x,y)}{F\left(
x\right) }\right) \left( 1+\frac{d(x,y)}{F\left( x\right) -d\left(
x,y\right) }\right) .  \label{ja}
\end{equation}%
Then $\underset{n\rightarrow \infty }{\lim }d\left( x_{n},x\right) =0$
implies $\underset{n\rightarrow \infty }{\lim }j\left( x_{n},x\right) =0$,
therefore the identity map $1_{X\setminus M}:\left( X\setminus M,d\right)
\rightarrow \left( X\setminus M,j\right) $\ is continuous. \newline
For every $y\in X$, $j\left( x,y\right) \geq \frac{1}{2}\log \left( 1+\frac{%
d(x,y)}{F\left( x\right) }\right) \left( 1+\frac{d(x,y)}{F\left( x\right)
+d\left( x,y\right) }\right) $, hence \newline
$j\left( x,y\right) \geq \log \left( 1+\frac{d(x,y)}{F\left( x\right)
+d\left( x,y\right) }\right) $, which implies 
\begin{equation}
d(x,y)\leq F(x)\frac{e^{j(x,y)}-1}{2-e^{j(x,y)}}\text{ if }j\left(
x,y\right) <2\text{. }  \label{jb}
\end{equation}%
If $\underset{n\rightarrow \infty }{\lim }j\left( x_{n},x\right) =0$, (\ref%
{jb}) shows that $\underset{n\rightarrow \infty }{\lim }d\left(
x_{n},x\right) =0$, therefore the identity $1_{X\setminus M}:\left(
X\setminus M,d\right) \rightarrow \left( X\setminus M,j\right) $\ is open,
hence it is a homeomorphism. \newline
\newline
For $f=1_{X}:\left( X,d\right) \rightarrow \left( X\setminus M,j\right) $ we
have $H_{f}\left( x,r\right) =\frac{\sup \left\{ j(x,y):d\left( x,y\right)
\leq r\right\} }{\inf \left\{ j(x,y):d\left( x,y\right) \geq r\right\} }$.
For $0<r<F(x)$, (\ref{ja}) implies $\sup \left\{ j(x,y):d\left( x,y\right)
\leq r\right\} \leq \frac{1}{2}\log \left( 1+\frac{r}{F\left( x\right) }%
\right) \left( 1+\frac{r}{F\left( x\right) -r}\right) $. \newline
For every $r>0$, $\inf \left\{ j(x,y):d\left( x,y\right) \geq r\right\} \geq 
\frac{1}{2}\log \left( 1+\frac{r}{F\left( x\right) }\right) \left( 1+\frac{r%
}{F\left( x\right) +r}\right) $. \newline
Then $\underset{r\rightarrow 0}{\lim \sup }H_{f}\left( x,r\right) \leq 
\underset{r\rightarrow 0}{\lim \sup }\frac{\log \left( 1+\frac{r}{F\left(
x\right) }\right) \left( 1+\frac{r}{F\left( x\right) -r}\right) }{\log
\left( 1+\frac{r}{F\left( x\right) }\right) \left( 1+\frac{r}{F\left(
x\right) +r}\right) }$. \newline
By L'Hospital' rule, $\underset{r\rightarrow 0}{\lim }\frac{\log \left( 1+%
\frac{r}{F\left( x\right) }\right) \left( 1+\frac{r}{F\left( x\right) -r}%
\right) }{\log \left( 1+\frac{r}{F\left( x\right) }\right) \left( 1+\frac{r}{%
F\left( x\right) +r}\right) }=\underset{r\rightarrow 0}{\lim }\frac{\frac{1}{%
F(x)+r}+\frac{1}{F(x)-r}}{\frac{2}{F(x)+2r}}=1$. It follows that the
homeomorphism $1_{X\setminus M}:\left( X\setminus M,d\right) \rightarrow
\left( X\setminus M,j\right) $\ is $1-$quasiconformal.
\end{proof}

\section{A generalization of the Dovgoshey-Hariri-Vuorinen metric}

Dovgoshey-Hariri-Vuorinen' s metric has been introduced in \cite{DHV2016}
and has its roots in the study of distance ratio metric.

\begin{theorem}
\label{DHVgeneralized}Let $\left( X,d\right) $ be a metric space and $M$ be
a closed nonempty proper subset of $X$. Let $F:\left( X\setminus M,d\right)
\rightarrow \left( 0,\infty \right) $ be a $1-$Lipschitz function. If $c\geq
2$ is a constant, then the function defined by $h_{c}\left( x,y\right) =\log
\left( 1+c\frac{d(x,y)}{\sqrt{F(x)F\left( y\right) }}\right) $ for every $%
x,y\in X\setminus M$ is a metric on $X\setminus M$.
\end{theorem}

\begin{proof}
Obviously, for all $x,y\in X\setminus M$, $h_{c}\left( x,y\right) \geq 0$, $%
h_{c}\left( x,y\right) =0$ is equivalent to $x=y$, and $h_{c}\left(
x,y\right) =h_{c}\left( y,x\right) $.

We will prove that for every $x,y,z\in G$ the triangle inequality $%
h_{c}\left( x,y\right) \leq h_{c}\left( x,z\right) +h_{c}\left( z,y\right) $
holds, which is equivalent to 
\begin{equation}
c\frac{d\left( x,z\right) d\left( z,y\right) }{F\left( z\right) \sqrt{%
F(x)F\left( y\right) }}+\frac{d(x,z)}{\sqrt{F(x)F\left( z\right) }}+\frac{%
d(z,y)}{\sqrt{F(y)F\left( z\right) }}\geq \frac{d(x,y)}{\sqrt{F(x)F\left(
y\right) }}.  \label{Form1}
\end{equation}

Since the left-hand side in (\ref{Form1}) is non-decreasing with respect to $%
c$, it suffices to consider the case $c=2$. Due to triangle inequality $%
d\left( x,y\right) \leq d(x,z)+d\left( z,y\right) $, it suffices to prove
that 
\begin{equation}
2\frac{d\left( x,z\right) d\left( z,y\right) }{F\left( z\right) \sqrt{%
F(x)F\left( y\right) }}+\frac{d(x,z)}{\sqrt{F(x)F\left( z\right) }}+\frac{%
d(z,y)}{\sqrt{F(y)F\left( z\right) }}\geq \frac{d(x,z)+d\left( z,y\right) }{%
\sqrt{F(x)F\left( y\right) }}.  \label{Suf1}
\end{equation}%
Denote $d\left( x,z\right) =a$, $d\left( z,y\right) =b$, $F\left( x\right)
=u $, $F\left( y\right) =v$ and $F\left( z\right) =w$. Using these
notations, inequality (\ref{Suf1}) is equivalent to 
\begin{equation}
2\frac{ab}{w}+a\sqrt{\frac{v}{w}}+b\sqrt{\frac{u}{w}}\geq a+b\text{.}
\label{Suf2}
\end{equation}%
Using the $1-$Lipschiz continuity of the function $F$, it follows that $%
w\leq \min \left\{ a+u,b+v\right\} $. We may assume that $a+u\leq b+v$,
otherwise we swap $x$ and $y$. Then 
\begin{equation}
w\leq a+u\leq b+v\text{, }  \label{Ineqe}
\end{equation}%
hence we have 
\begin{equation}
2\frac{ab}{w}+a\sqrt{\frac{v}{w}}+b\sqrt{\frac{u}{w}}\geq 2\frac{ab}{a+u}+a%
\sqrt{\frac{v}{a+u}}+b\sqrt{\frac{u}{a+u}}\text{.}  \label{Est3}
\end{equation}%
We check if 
\begin{equation}
2\frac{ab}{a+u}+a\sqrt{\frac{v}{a+u}}+b\sqrt{\frac{u}{a+u}}\geq a+b.
\label{Est4}
\end{equation}%
\textbf{Case 1. }Assume that $b\geq a+u$. Then 
\begin{eqnarray*}
2\frac{ab}{a+u}+a\sqrt{\frac{v}{a+u}}+b\sqrt{\frac{u}{a+u}} &\geq &2\frac{ab%
}{a+u}+b\sqrt{\frac{u}{a+u}} \\
&\geq &a+\frac{ab}{a+u}+b\sqrt{\frac{u}{a+u}}
\end{eqnarray*}%
But 
\begin{equation*}
a+\frac{ab}{a+u}+b\sqrt{\frac{u}{a+u}}-\left( a+b\right) =b\sqrt{\frac{u}{a+u%
}}\left( 1-\sqrt{\frac{u}{a+u}}\right) \geq 0.
\end{equation*}%
The latter inequalities show that (\ref{Est4}) holds in this case. \newline
\textbf{Case 2. }Assume that\textbf{\ }$b<a+u$. Denote $\lambda =\frac{b}{a+u%
}$. Note that $\lambda \in \lbrack 0,1)$. \newline
By (\ref{Ineqe}), $v\geq a+u-b$, which implies the following lower bound for
the difference between the left-hand side and the right-hand side in (\ref%
{Est4}): 
\begin{eqnarray*}
&&2\frac{ab}{a+u}+a\sqrt{\frac{v}{a+u}}+b\sqrt{\frac{u}{a+u}}-\left(
a+b\right) \\
&\geq &a\lambda +a\sqrt{1-\lambda }+\lambda \sqrt{u(a+u)}-a-\lambda u \\
&=&\lambda \sqrt{u}\left( \sqrt{a+u}-\sqrt{u}\right) +a\sqrt{1-\lambda }%
\left( 1-\sqrt{1-\lambda }\right) \geq 0,
\end{eqnarray*}%
hence (\ref{Est4}) holds.

By (\ref{Est4}) and (\ref{Est3}) we get (\ref{Suf2}), which is equivalent to
(\ref{Suf1}). Now using triangle inequality $d\left( x,y\right) \leq
d(x,z)+d\left( z,y\right) $ we see that (\ref{Suf1}) implies (\ref{Form1})
and the claim is proved.
\end{proof}

\begin{corollary}
Let $\left( X,d\right) $ be a metric space, $G\subset X$ be an open set with
non-empty boundary and $c\geq 2$ be a constant. Then the
Dovgoshey-Hariri-Vuorinen's function defined by $h_{G,c}\left( x,y\right)
=\log \left( 1+c\frac{d(x,y)}{\sqrt{\delta (x)\delta \left( y\right) }}%
\right) $ for every $x,y\in G$, is a metric. Here $\delta \left( x\right) =%
\mathrm{dist}\left( x,\partial G\right) $, $x\in X$.
\end{corollary}

\begin{remark}
Consider $X=\mathbb{R}^{2}$ with the Euclidean metric. If $G=\mathbb{B}^{2}$
is the unit disk, it was shown in \cite{DHV2016} that $h_{c}$ is not a
metric if $0<c<2$. Very recently, Fujimura and Vuorinen \cite{FujiVu} proved
in the case $G=\mathbb{H}^{2}$ (the upper half plane) that the function $%
h_{c}$ is a metric for every $c\geq 1$.
\end{remark}

\begin{theorem}
Let $\left( X,d\right) $ be a metric space and $M$ be a closed nonempty
proper subset of $X$. Let Let $F:\left( X\setminus M,d\right) \rightarrow
\left( 0,\infty \right) $ be a $1-$Lipschitz function. Let $c>0$. The
generalized Dovgoshey-Hariri-Vuorinen function $h_{c}\left( x,y\right) =\log
\left( 1+c\frac{d(x,y)}{\sqrt{F(x)F\left( y\right) }}\right) $, $x,y\in
X\setminus M$ satisfies the Gromov hyperbolicity condition with a Gromov
constant $\delta \leq \log \left( 2+\frac{1}{c}\right) $.
\end{theorem}

\begin{proof}
We show that the proof of Theorem 1.1 in \cite{Zhou} remains valid after
relaxing the assumptions of this theorem, by replacing: the Euclidean space $%
\mathbb{R}^{n}$ by an arbitrary metric space $\left( X,d\right) $, $\partial
G$ by an arbitrary nonempty closed proper subset $M$ of $X$, the domain of
definition $G$ by $X\setminus M$, and $\delta \left( \cdot \right) =\mathrm{%
dist}\left( \cdot ,\partial G\right) $ by an arbitrary $1-$Lipschitz
function $F$. Note that we do not need to assume neither that $c\geq 2$ nor
that $h_{c}$ is a metric.

Denoting $\lambda \left( x,y\right) =cd(x,y)+\sqrt{F(x)F(y)}$, the
inequality $h_{c}(x,z)+h_{c}(y,w)\leq \max \left\{
h_{c}(x,w)+h_{c}(y,z),h_{c}(x,y)+h_{c}(z,w)\right\} +2\delta $ is equivalent
to 
\begin{equation}
\lambda \left( x,z\right) \lambda \left( y,w\right) \leq e^{2\delta }\max
\left\{ \lambda \left( x,w\right) \lambda \left( y,z\right) \text{, }\lambda
\left( x,y\right) \lambda \left( z,w\right) \right\} .  \label{Gromov_DHV}
\end{equation}%
For arbitrary $x,y,z\in X$, as $F$ is $1-$Lipschitz, we have $F(x)\leq \sqrt{%
F(x)F(z)}+d\left( x,z\right) $ and $F(y)\leq \sqrt{F(z)F(y)}+d\left(
z,y\right) $, hence $\sqrt{F(x)F(y)}\leq \frac{F(x)+F(y)}{2}\leq \frac{1}{2}%
(d\left( x,z\right) +d\left( z,y\right) )+\frac{1}{2}\left( \sqrt{F(x)F(z)}+%
\sqrt{F(z)F(y)}\right) $. \newline
By the triangle inequality, $cd\left( x,y\right) \leq c\left( d\left(
x,z\right) +d\left( z,y\right) \right) $. Adding the latter two inequalities
we get 
\begin{eqnarray*}
\lambda \left( x,y\right) &\leq &\frac{2c+1}{2c}c\left( d\left( x,z\right)
+d\left( z,y\right) \right) +\frac{1}{2}\left( \sqrt{F(x)F(z)}+\sqrt{F(z)F(y)%
}\right) \\
&\leq &\frac{2c+1}{2c}\left( \lambda \left( x,z\right) +\lambda \left(
z,y\right) \right) \leq \frac{2c+1}{c}\max \left\{ \lambda \left( x,z\right)
,\lambda \left( z,y\right) \right\} .
\end{eqnarray*}%
Fix $x,y,z,w$ $\in X$. As in the proof of Theorem \ref{GOmetric_Gen}, we may
assume that $\lambda \left( x,w\right) \leq \min \left\{ \lambda \left(
x,y\right) ,\lambda \left( y,z\right) ,\lambda \left( z,w\right) \right\} $
by swapping $x$ and $z$ and/ or $y$ and $w$. Then 
\begin{eqnarray*}
\lambda \left( x,z\right) &\leq &\frac{2c+1}{c}\max \left\{ \lambda \left(
x,w\right) ,\lambda \left( w,z\right) \right\} =\frac{2c+1}{c}\lambda \left(
z,w\right) \text{ and } \\
\lambda \left( y,w\right) &\leq &\frac{2c+1}{c}\max \left\{ \lambda \left(
y,x\right) ,\lambda \left( x,w\right) \right\} =\frac{2c+1}{c}\lambda \left(
x,y\right) .
\end{eqnarray*}%
Multiplying the latter two inequalities we get (\ref{Gromov_DHV}) with $%
e^{2\delta }=\left( \frac{2c+1}{c}\right) ^{2}$.
\end{proof}

\begin{theorem}
Let $\left( X,d\right) $ be a metric space and $h_{c}$ be the generalized
Dovgoshey-Hariri-Vuorinen metric from Theorem \ref{DHVgeneralized}, where $%
c\geq 2$. Then the identity map $1_{X\setminus M}:\left( X\setminus
M,d\right) \rightarrow \left( X\setminus M,h_{c}\right) $ is $1-$%
quasiconformal.
\end{theorem}

\begin{proof}
Let $x\in X$. For every $y\in X$ such that $d(x,y)<F(x)$ we have 
\begin{equation}
h_{c}\left( x,y\right) \leq \log \left( 1+c\frac{d(x,y)}{\sqrt{F(x)\left(
F\left( x\right) -d(x,y\right) )}}\right) ,  \label{Ha}
\end{equation}%
hence the identity \ map $1_{X\setminus M}:\left( X\setminus M,d\right)
\rightarrow \left( X\setminus M,h_{c}\right) $ is continuous. For every $%
y\in X$, 
\begin{equation}
h_{c}\left( x,y\right) \geq \log \left( 1+c\frac{d(x,y)}{\sqrt{F(x)\left(
F\left( x\right) +d(x,y\right) )}}\right) ,  \label{Hb1}
\end{equation}%
which is equivalent to $\frac{d(x,y)}{\sqrt{F\left( x\right) +d(x,y)}}\leq 
\frac{\sqrt{F(x)}}{c}\left( e^{h_{c}(x,y)}-1\right) $. The inequality $\frac{%
t}{\sqrt{a+t}}\leq b$, where $a,b,t>0$ holds if and only if $0<t\leq \frac{b%
}{2}\left( b+\sqrt{b^{2}+4a}\right) $. Then for every $x\neq y$ in $X$ ,
denoting $H_{c}(x,y)=\left( e^{h_{c}(x,y)}-1\right) $ we have 
\begin{equation}
d\left( x,y\right) \leq \frac{F(x)}{2c^{2}}H_{c}(x,y)\left( H_{c}(x,y)+\sqrt{%
\left( H_{c}(x,y)\right) ^{2}+4c^{2}}\right) .  \label{Hb}
\end{equation}%
If $\underset{n\rightarrow \infty }{\lim }h_{c}\left( x_{n},x\right) =0$, (%
\ref{Hb}) shows that $\underset{n\rightarrow \infty }{\lim }d\left(
x_{n},x\right) =0$, therefore the identity \ map $1_{X\setminus M}:\left(
X\setminus M,d\right) \rightarrow \left( X\setminus M,h_{c}\right) $ is open
and we conclude that it is a homeomorphism. \newline
For $0<r<F(x)$, (\ref{Ha}) implies $\sup \left\{ h_{c}(x,y):d\left(
x,y\right) \leq r\right\} \leq \log \left( 1+c\frac{r}{\sqrt{F(x)(F\left(
x\right) -r)}}\right) $, since the function $\frac{t}{\sqrt{a-t}}$ is
increasing for $t\in (0,a)$, where $a>0$. \newline
For every $r>0$, (\ref{Hb1}) implies $\inf \left\{ h_{c}(x,y):d\left(
x,y\right) \geq r\right\} \geq \log \left( 1+c\frac{r}{\sqrt{F(x)(F\left(
x\right) +r)}}\right) $, since the function $\frac{t}{\sqrt{a+t}}$ is
increasing for $t\in (0,\infty )$, where $a>0$. \newline
\newline
Then $\underset{r\rightarrow 0}{\lim \sup }H_{f}\left( x,r\right) \leq 
\underset{r\rightarrow 0}{\lim \sup }\frac{\log \left( 1+c\frac{r}{\sqrt{%
F(x)(F\left( x\right) -r)}}\right) }{\log \left( 1+c\frac{r}{\sqrt{%
F(x)(F\left( x\right) +r)}}\right) }$. \newline
By L'Hospital' rule, $\underset{r\rightarrow 0}{\lim }\frac{\log \left( 1+c%
\frac{r}{\sqrt{F(x)(F\left( x\right) -r)}}\right) }{\log \left( 1+c\frac{r}{%
\sqrt{F(x)(F\left( x\right) +r)}}\right) }=\underset{r\rightarrow 0}{\lim }%
\frac{\frac{d}{dr}\left( \frac{r}{\sqrt{F\left( x\right) -r}}\right) }{\frac{%
d}{dr}\left( \frac{r}{\sqrt{F\left( x\right) +r}}\right) }=1$. It follows
that the homeomorphism $1_{X\setminus M}:\left( X\setminus M,d\right)
\rightarrow \left( X\setminus M,h_{c}\right) $ is $1-$quasiconformal.
\end{proof}

\begin{remark}
In the proof of the above theorem we did not use the assumption $c\geq 2$,
but only $c>0$.
\end{remark}

\section{A generalization of Nikolov-Andreev metric}

Nikolov-Andreev metric has been introduced in \cite{Nikolov}. The triangle
inequality for a generalization of Nikolov-Andreev metric is proved in \cite[%
Proposition 11]{Nikolov}. However, we give below an explicit proof, for the
sake of completness and in order to extend the domain where the metric is
defined.

\begin{theorem}
\label{genNA}Let $\left( X,d\right) $ be a metric space and $M$ be a closed
nonempty proper subset of $X$. Let $F:\left( X\setminus M,d\right)
\rightarrow \left( 0,\infty \right) $ be a $1-$Lipschitz function. Then the
function defined by $i\left( x,y\right) =2\log \frac{F(x)+F(y)+d(x,y)}{2%
\sqrt{F(x)F\left( y\right) }}$ for every $x,y\in X\setminus M$ is a metric.
\end{theorem}

\begin{proof}
For every $x,y\in X\setminus M$ we have $i\left( x,y\right) \geq 0$, since $%
F(x),F(y)>0$ and $F(x)+F(y)\geq 2\sqrt{F(x)F\left( y\right) }$, $i\left(
x,y\right) =0$ if and only if $x=y$, as well as $i\left( x,y\right) =i\left(
y,x\right) $. \newline
Denote $d\left( x,z\right) =a$, $d\left( z,y\right) =b$, $F\left( x\right)
=u $, $F\left( y\right) =v$ and $F\left( z\right) =w$. The triangle
inequality for the function $i$ is equivalent to 
\begin{equation*}
\frac{a+u+w}{2\sqrt{uw}}\frac{b+v+w}{2\sqrt{vw}}\geq \frac{d(x,y)+c+d}{2%
\sqrt{cd}}.
\end{equation*}%
By the triangle inequality for the distance $d(\cdot ,\cdot )$, it suffices
to prove that 
\begin{equation*}
\frac{a+u+w}{2\sqrt{uw}}\frac{b+v+w}{2\sqrt{vw}}\geq \frac{a+b+u+v}{2\sqrt{uv%
}}\text{,}
\end{equation*}%
that is equivalent to $\left[ w+\left( a+u\right) \right] \cdot \left[
w+\left( b+v\right) \right] \geq 2w\left( a+u+b+v\right) $. i.e. to 
\begin{equation*}
\left[ w-\left( a+u\right) \right] \cdot \left[ w-\left( b+v\right) \right]
\geq 0\text{.}
\end{equation*}%
The latter inequality holds since the $1-$Lipschiz continuity of the
function $F$ implies $w\leq \min \left\{ a+u,b+v\right\} $.
\end{proof}

In \cite[Theorem 1.1]{Luo} it was proved that the Nikolov-Andreev metric $%
i_{\Omega }$ in a proper subdomain $\Omega \subset \mathbb{R}^{n}$, defined
by $i_{\Omega }\left( x,y\right) =2\log \frac{d_{\Omega }(x)+d_{\Omega
}(y)+d(x,y)}{2\sqrt{d_{\Omega }(x)d_{\Omega }(y)}}$ for $x,y\in \Omega $, is
Gromov hyperbolic with a Gromov constant $\delta =\log 15$.

\begin{theorem}
Let $\left( X,d\right) $ be a metric space and $M$ be a closed nonempty
proper subset of $X$. Let $F:\left( X\setminus M,d\right) \rightarrow
(0,\infty )$ be a $1-$Lipschitz function. Then the generalized
Nikolov-Andreev metric defined by $i\left( x,y\right) =2\log \frac{%
F(x)+F(y)+d(x,y)}{2\sqrt{F(x)F\left( y\right) }}$ for $x,y\in X\setminus M$
is Gromov hyperbolic with a Gromov constant $\delta \leq \log 9$.
\end{theorem}

\begin{proof}
Let $x,y,z,w\in X\setminus M$. Denoting $\nu \left( x,y\right)
=F(x)+F(y)+d(x,y)$, the inequality $i(x,z)+i(y,w)\leq \max \left\{
i(x,w)+i(y,z),i(x,y)+i(z,w)\right\} +2\delta $ is equivalent to 
\begin{equation}
\nu \left( x,z\right) \nu \left( y,w\right) \leq e^{\delta }\max \left\{ \nu
\left( x,w\right) \nu \left( y,z\right) \text{, }\nu \left( x,y\right) \nu
\left( z,w\right) \right\} .  \label{Gromov_Nikolov}
\end{equation}%
As $F$ is $1-$Lipschitz, $F(x)\leq F(z)+d\left( x,z\right) $ and $F(y)\leq
F(z)+d\left( z,y\right) $, hence $F(x)+F(y)\leq 2F(z)+d(x,z)+d(z,y)$. This
inequality implies 
\begin{equation*}
v\left( x,z\right) +v\left( z,y\right) =F(x)+F(y)+2F(z)+d(x,z)+d(z,y)\geq
2\left( F(x)+F(y)\right) ,
\end{equation*}%
hence $F(x)+F(y)\leq \frac{1}{2}\left( v\left( x,z\right) +v\left(
z,y\right) \right) $. Using the triangle inequality for $d$ we see that $%
d\left( x,y\right) <\nu \left( x,z\right) +\nu \left( z,y\right) $. Adding
the latter two inequalities we obtain for every $x,y,z$ $\in X$ the
inequalities 
\begin{equation*}
\nu \left( x,y\right) <\frac{3}{2}\left( \nu \left( x,z\right) +\nu \left(
z,y\right) \right) \leq 3\max \left\{ \nu \left( x,z\right) ,\nu \left(
z,y\right) \right\} .
\end{equation*}%
Fix $x,y,z,w$ $\in X$. As in the proof of Theorem \ref{GOmetric_Gen}, we may
assume that $\nu \left( x,w\right) \leq \min \left\{ \nu \left( x,y\right)
,\nu \left( y,z\right) ,\nu \left( z,w\right) \right\} $ by swapping $x$ and 
$z$ and/ or $y$ and $w$. Then $\nu \left( x,z\right) \leq 3\max \left\{ \nu
\left( x,w\right) ,\nu \left( w,z\right) \right\} =3\nu \left( z,w\right) $
and $\nu \left( y,w\right) \leq 3\max \left\{ \nu \left( y,x\right) ,\nu
\left( x,w\right) \right\} =3\nu \left( x,y\right) $. Multiplying the latter
two inequalities we get (\ref{Gromov_Nikolov}) with $e^{\delta }=9$.
\end{proof}

\begin{theorem}
\label{NAQC}Let $\left( X,d\right) $ be a metric space and $i$ be the
generalized Nikolov-Andreev metric from Theorem \ref{genNA}. Then the
identity map $1_{X\setminus M}:\left( X\setminus M,d\right) \rightarrow
\left( X\setminus M,i\right) $ is $3-$quasiconformal.
\end{theorem}

\begin{proof}
Let $x\in X$. For every $y\in X$ such that $d(x,y)<F(x)$ we have 
\begin{equation}
i\left( x,y\right) \leq 2\log \frac{F(x)+d(x,y)}{\sqrt{F(x)\left( F\left(
x\right) -d\left( x,y\right) \right) }},  \label{ia}
\end{equation}%
hence the identity map $1_{X\setminus M}:\left( X\setminus M,d\right)
\rightarrow \left( X\setminus M,i\right) $ is continuous. For every $y\in X$%
, using the mean inequality and the $1-$Lipschitz property of $f$ we get 
\begin{equation}
i\left( x,y\right) \geq 2\log \left( 1+\frac{d(x,y)}{2\sqrt{F(x)(F(x)+d(x,y))%
}}\right) .  \label{ib}
\end{equation}%
The inequality $\frac{t}{\sqrt{a+t}}\leq b$, where $a,b>0$ and $t>0$ holds
if and only if $0<t\leq \frac{b}{2}\left( b+\sqrt{b^{2}+4a}\right) $. Then
for every $x\neq y$ in $X$ we have 
\begin{equation*}
d\left( x,y\right) \leq 2F(x)\left( e^{\frac{i(x,y)}{2}}-1\right) \left[
\left( e^{\frac{i(x,y)}{2}}-1\right) +\sqrt{\left( e^{\frac{i(x,y)}{2}%
}-1\right) ^{2}+1}\right] .
\end{equation*}%
If $\underset{n\rightarrow \infty }{\lim }i\left( x_{n},x\right) =0$, the
above inequality shows that $\underset{n\rightarrow \infty }{\lim }d\left(
x_{n},x\right) =0$, therefore $1_{X\setminus M}:\left( X\setminus M,d\right)
\rightarrow \left( X\setminus M,i\right) $ is open, hence it is a
homeomorphism \newline
For $0<r<F(x)$, (\ref{ia}) implies $\sup \left\{ i(x,y):d\left( x,y\right)
\leq r\right\} \leq 2\log \left( \frac{F(x)+r}{\sqrt{F(x)(F\left( x\right)
-r)}}\right) $, as the function $\frac{a+t}{\sqrt{a-t}}$ is increasing for $%
t\in (0,a)$, where $a>0$. For every $r>0$, $\inf \left\{ i(x,y):d\left(
x,y\right) \geq r\right\} \geq 2\log \left( 1+\frac{r}{2\sqrt{F(x)(F(x)+r)}}%
\right) $, as the function $\frac{t}{\sqrt{a+t}}$ is increasing for $t\in
(0,\infty )$, where $a>0$. \newline
Then $\underset{r\rightarrow 0}{\lim \sup }H_{f}\left( x,r\right) \leq 
\underset{r\rightarrow 0}{\lim \sup }\frac{\log \left( \frac{F(x)+r}{\sqrt{%
F(x)(F\left( x\right) -r)}}\right) }{\log \left( 1+\frac{r}{2\sqrt{%
F(x)(F(x)+r)}}\right) }$. \newline
Applying L'Hospital's rule we get 
\begin{equation*}
\underset{r\rightarrow 0}{\lim }\frac{\log \left( \frac{F(x)+r}{\sqrt{%
F(x)(F\left( x\right) -r)}}\right) }{\log \left( 1+\frac{r}{2\sqrt{%
F(x)(F(x)+r)}}\right) }=\underset{r\rightarrow 0}{\lim }\frac{\frac{1}{F(x)+r%
}+\frac{1}{2}\frac{1}{F(x)-r}}{\frac{1}{r+2\sqrt{F(x)(F(x)+r)}}\left( 1+%
\sqrt{\frac{F(x)}{F(x)+r}}\right) -\frac{1}{2}\frac{1}{F(x)+r}}=3.
\end{equation*}%
It follows that the homeomorphism $1_{X}:\left( X,d\right) \rightarrow
\left( X\setminus M,i\right) $ is $3-$quasiconformal.
\end{proof}

\section{A generalization of Ibragimov metric}

Let $\left( X,d\right) $ be a metric space and $M$ be a nonempty closed
proper subset of $X$. Denote $d_{M}(x)=$\textrm{dist}$\left( x,M\right) $
for every $x\in X$.

Ibragimov defined the following function $u_{X}\left( x,y\right) =2\log 
\frac{d(x,y)+\max \left\{ d_{M}(x),d_{M}(y)\right\} }{\sqrt{d_{M}(x)d_{M}(y)}%
}$ for $x,y\in X\setminus M$ and proved in \cite[Theorem 2.1]{Ibra} that $%
\left( X\setminus M,u_{X}\right) $ is a \ Gromov hyperbolic metric space
with a Gromov constant $\delta \leq \log 4$. Moreover, the identity map
between the metric spaces $\left( X\setminus M,d\right) $ and $\left(
X\setminus M,u_{X}\right) $ is $5-$quasiconformal and if the space $\left(
X,d\right) $ is complete, then $\left( X\setminus M,u_{X.}\right) $ is so.
We will slightly extend \cite[Theorem 2.1]{Ibra}.

\begin{theorem}
Let $\left( X,d\right) $ be a metric space and $M$ be a nonempty closed
proper subset of $X$. Let $F:X\setminus M\rightarrow (0,\infty )$ be a
function. Define $v\left( x,y\right) =2\log \frac{d(x,y)+\max \left\{
F(x),F(y)\right\} }{\sqrt{F(x)F(y)}}$ for $x,y\in X\setminus M$. Then $%
\left( X\setminus M,v\right) $ is a Gromov hyperbolic metric space with a
Gromov constant $\delta \leq \log 4$ and the identity map $1_{X\setminus
M}:\left( X\setminus M,d\right) \rightarrow \left( X\setminus M,v\right) $
is open. If $F$ is $1-$Lipschitz on $\left( X\setminus M,d\right) $, then
the identity map $1_{X\setminus M}:\left( X\setminus M,d\right) \rightarrow
\left( X\setminus M,v\right) $ is $\frac{5}{2}-$quasiconformal. If $F$ has a
continuous extension $\widetilde{F}$ to $\left( X,d\right) $ with $%
\widetilde{F}\left( x\right) =0$ for every $x\in M$ and $\left( X,d\right) $
is complete, then $\left( X\setminus M,v\right) $ is also a complete metric
space.
\end{theorem}

\begin{proof}
We follow the lines of the proof of \cite[Theorem 2.1]{Ibra}.

a) First we prove that $v$ is a metric on $X\setminus M$. For every $x,y\in
X\setminus M$ the symmetry $v(x,y)=v(y,x)$ is obvious, while $\max \left\{
F(x),F(y)\right\} \geq \sqrt{F(x)F(y)}>0$ and $d\left( x,y\right) \geq 0$
implies $v\left( x,y\right) \geq 0$, with $v\left( x,y\right) =0$ if and
only if $x=y$.\newline
The triangle inequality $v\left( x,y\right) \leq v\left( x,z\right) +v\left(
z,y\right) $ is equivalent to 
\begin{eqnarray}
&&F(z)(d(x,y)+\max \left\{ F(x),F(y)\right\}  \label{TIV} \\
&\leq &\left( d(x,z)+\max \left\{ F(x),F(z)\right\} \right) \cdot \left(
d\left( z,y\right) +\max \left\{ F(z),F(y)\right\} \right) .  \notag
\end{eqnarray}%
\newline
The triangle inequality $d\left( x,y\right) \leq d(x,z)+d(z,y)$ implies \ 
\begin{eqnarray}
&&F(z)(d(x,y)+\max \left\{ F(x),F(y)\right\} )  \label{TID} \\
&\leq &F(z)d(x,z)+F(z)d(z,y)+F(z)\max \left\{ F(x),F(y)\right\} .  \notag
\end{eqnarray}%
But $F(z)\leq \max \left\{ F(z),F(y)\right\} $ and $F(z)\leq \max \left\{
F(x),F(z)\right\} $, hence $F(z)d(x,z)+F(z)d(z,y)\leq d(x,z)\max \left\{
F(z),F(y)\right\} +d(z,y)\max \left\{ F(x),F(z)\right\} $. \newline
In addition, $F(z)\max \left\{ F(x),F(y)\right\} \leq \max \left\{
F(x),F(z)\right\} \cdot \max \left\{ F(z),F(y)\right\} $. \newline
Adding the latter two inequalities we get 
\begin{eqnarray}
&&F(z)d(x,z)+F(z)d(z,y)+F(z)\max \left\{ F(x),F(y)\right\}  \label{AD2} \\
&\leq &\left( d(x,z)+\max \left\{ F(x),F(z)\right\} \right) \cdot \left(
d\left( z,y\right) +\max \left\{ F(z),F(y)\right\} \right) -d(x,z)d\left(
z,y\right)  \notag
\end{eqnarray}%
By (\ref{TID}) and (\ref{AD2}) we get (\ref{TIV}). \newline
b) Let $x,y,z,w\in X\setminus M$. We will prove that 
\begin{equation}
v(x,z)+v(y,w)\leq \max \left\{ v(x,w)+v(y,z),v(x,y)+v(z,w)\right\} +2\log 4.
\label{GHV}
\end{equation}%
As in \cite{Ibra}, denote $\mu \left( x,y\right) =d(x,y)+\max \left\{
F(x),F(y)\right\} $ for $x,y\in X\setminus M$. We see that $\mu $ is a
metric on $X\setminus M$, as the positivity of $F$ implies $\max \left\{
F(x),F(y)\right\} \leq \max \left\{ F(x),F(z)\right\} +\max \left\{
F(z),F(y)\right\} $and $d$ is a metric on $X$. \newline
Inequality (\ref{GHV}) is equivalent to 
\begin{equation}
\mu \left( x,z\right) \mu \left( y,w\right) \leq 4\max \left\{ \mu \left(
x,w\right) \mu \left( y,z\right) ,\mu \left( x,y\right) \mu \left(
z,w\right) \right\} .  \label{GHV2}
\end{equation}%
As in the proof of Theorem \ref{GOmetric_Gen}, we may assume that $\mu
\left( x,w\right) \leq \min \left\{ \mu \left( x,y\right) ,\mu \left(
y,z\right) ,\mu \left( z,w\right) \right\} $ by swapping $x$ and $z$ and/ or 
$y$ and $w$. Then 
\begin{equation*}
\mu \left( x,z\right) \mu \left( y,w\right) \leq \left( \mu \left(
x,w\right) +\mu \left( z,w\right) \right) \cdot \left( \mu \left( x,w\right)
+\mu \left( x,y\right) \right) \leq 2\mu \left( z,w\right) \cdot 2\mu \left(
x,y\right) \text{,}
\end{equation*}%
hence inequality (\ref{GHV2}) holds.

c) The estimates (2.3) and (2.4) from \cite{Ibra} can be extended from $%
u_{X}\left( x,y\right) $ to $v\left( x,y\right) $ without assuming that $F$
is $1-$Lipschitz. For every $x,y\in X\setminus M$ we have 
\begin{equation}
v(x,y)\geq 2\log \frac{\max \left\{ F(x),F(y)\right\} }{\sqrt{F(x)F\left(
y\right) }}=\log \left\vert \frac{F(x)}{F(y)}\right\vert \text{ and }
\label{Ib1}
\end{equation}%
\begin{equation}
v(x,y)\geq \log \left( 1+\frac{d(x,y)}{F(x)}\right) \left( 1+\frac{d(x,y)}{%
F(y)}\right) \geq \log \left( 1+\frac{d(x,y)}{F(x)}\right) .  \label{Ib2}
\end{equation}%
The latter estimate implies $d\left( x,y\right) \leq F(x)\left(
e^{v(x,y)}-1\right) $, hence $\underset{n\rightarrow \infty }{\lim }v\left(
x_{n},x\right) =0$ implies $\underset{n\rightarrow \infty }{\lim }d\left(
x_{n},x\right) =0$, which shows that the identity map $1_{X\setminus
M}:\left( X\setminus M,d\right) \rightarrow \left( X\setminus M,v\right) $
is open.

d) Now assume that $F:X\setminus M\rightarrow (0,\infty )$ is $1-$Lipschitz.
Let $x,y\in X\setminus M$. Since $F(y)\leq F(x)+d\left( x,y\right) $, it
follows as in \cite[(2.6)]{Ibra} that for $d\left( x,y\right) <F(x)$ we have 
\begin{equation}
v(x,y)\leq \log \frac{\left( F(x)+2d\left( x,y\right) \right) ^{2}}{%
F(x)\left( F(x)-d\left( x,y\right) \right) }.\text{ }  \label{Iba}
\end{equation}%
By inequality (\ref{Iba}), $\underset{n\rightarrow \infty }{\lim }d\left(
x_{n},x\right) =0$ implies $\underset{n\rightarrow \infty }{\lim }v\left(
x_{n},x\right) =0$, hence $1_{X\setminus M}:\left( X\setminus M,d\right)
\rightarrow \left( X\setminus M,v\right) $ is continuous, and being open, is
a homeomorphism. \newline
Note that if $F$ is $1-$Lipschitz we may refine the estimates from below for 
$v\left( x,y\right) $: 
\begin{equation}
v(x,y)\geq 2\log \left( 1+\frac{d(x,y)}{\sqrt{F(x)F(y)}}\right) \geq 2\log
\left( 1+\frac{d(x,y)}{\sqrt{F(x)\left( F(x)+d\left( x,y\right) \right) }}%
\right) \text{.}  \label{IbB}
\end{equation}%
As in the proof of Theorem \ref{NAQC}, from the above inequality we obtain
for $x\neq y$ the following inequality: 
\begin{equation*}
d(x,y)\leq \frac{1}{2}F(x)\left( e^{\frac{v(x,y)}{2}}-1\right) \left[ \left(
e^{\frac{v(x,y)}{2}}-1\right) +\sqrt{\left( e^{\frac{v(x,y)}{2}}-1\right)
^{2}+4}\right] ,
\end{equation*}%
which shows that $\underset{n\rightarrow \infty }{\lim }v\left(
x_{n},x\right) =0$ implies $\underset{n\rightarrow \infty }{\lim }d\left(
x_{n},x\right) =0$.

Let $x\in X$. For $0<r<F(x)$, (\ref{Iba}) implies $\sup \left\{
v(x,y):d\left( x,y\right) \leq r\right\} \leq \log \frac{\left(
F(x)+2r\right) ^{2}}{F(x)\left( F(x)-r\right) }$, since the function $\frac{%
(a+t)^{2}}{a-t}$ is increasing for $t\in \left( 0,a\right) $, where $a>0$.
For every $r>0$, (\ref{Ib2}) implies $\inf \left\{ v(x,y):d\left( x,y\right)
\geq r\right\} \geq \log \left( 1+\frac{r}{F(x)}\right) $. Then $H_{f}\left(
x,r\right) \leq \frac{\log \frac{\left( F(x)+2r\right) ^{2}}{F(x)\left(
F(x)-r\right) }}{\log \left( 1+\frac{r}{F(x)}\right) }$ for $0<r<F(x)$,
therefore 
\begin{equation*}
\underset{r\rightarrow 0}{\lim \sup }H_{f}\left( x,r\right) \leq \underset{%
r\rightarrow 0}{\lim \sup }\frac{\log \frac{\left( F(x)+2r\right) ^{2}}{%
F(x)\left( F(x)-r\right) }}{\log \left( 1+\frac{r}{F(x)}\right) }=\underset{%
r\rightarrow 0}{\lim }\frac{\log \frac{\left( F(x)+2r\right) ^{2}}{%
F(x)\left( F(x)-r\right) }}{\log \left( 1+\frac{r}{F(x)}\right) }=\underset{%
r\rightarrow 0}{\lim }\frac{\frac{4}{F(x)+2r}+\frac{1}{F(x)-r}}{\frac{1}{%
F(x)+r}}=5\text{. }
\end{equation*}%
As in \cite[Theorem 2.1]{Ibra}, the identity map $1_{X}:\left( X\setminus
M,d\right) \rightarrow \left( X\setminus M,v\right) $ is $5$-quasiconformal.

Moreover, (\ref{IbB}) implies $\inf \left\{ v(x,y):d\left( x,y\right) \geq
r\right\} \geq 2\log \left( 1+\frac{r}{\sqrt{F(x)\left( F(x)+r\right) }}%
\right) $, since the function $\frac{t}{\sqrt{a+t}}$ is increasing for $t\in
\left( 0,\infty \right) $ whenever $a>0$. Then $H_{f}\left( x,r\right) \leq 
\frac{\log \frac{\left( F(x)+2r\right) ^{2}}{F(x)\left( F(x)-r\right) }}{%
2\log \left( 1+\frac{r}{\sqrt{F(x)\left( F(x)+r\right) }}\right) }$ for $%
0<r<F(x)$, therefore 
\begin{equation*}
\underset{r\rightarrow 0}{\lim \sup }H_{f}\left( x,r\right) \leq \underset{%
r\rightarrow 0}{\lim }\frac{\log \frac{\left( F(x)+2r\right) ^{2}}{%
F(x)\left( F(x)-r\right) }}{2\log \left( 1+\frac{r}{\sqrt{F(x)\left(
F(x)+r\right) }}\right) }=\underset{r\rightarrow 0}{\lim }\frac{\frac{4}{%
F(x)+2r}+\frac{1}{F(x)-r}}{\frac{2F(x)+r}{\left( F(x)+r\right) \left( r+%
\sqrt{F(x)\left( F(x)+r\right) }\right) }}=\frac{5}{2}\text{. }
\end{equation*}

e) Assuming that $F$ has a continuous extension $\widetilde{F}$ to $X$ with $%
\widetilde{F}\left( x\right) =0$ for every $x\in M$ and that $\left(
X,d\right) $ is complete, we prove that $\left( X\setminus M,v\right) $ is
complete. Let $\left( x_{n}\right) _{n\geq 1}$ be a Cauchy sequence in $%
\left( X\setminus M,v\right) $. By (\ref{Ib1}), $\left( \log F\left(
x_{n}\right) \right) _{n\geq 1}$ is a Cauchy sequence in $\mathbb{R}$. Then $%
\left( \log F\left( x_{n}\right) \right) _{n\geq 1}$ is bounded: $\alpha =%
\underset{n\geq 1}{\inf }$ $\log F(x_{n})\in \mathbb{R}$ and $\beta =%
\underset{n\geq 1}{\sup }$ $\log F(x_{n})\in \mathbb{R}$. Then $e^{\alpha }=%
\underset{n\geq 1}{\inf }$ $F(x_{n})\leq \underset{n\geq 1}{\sup }%
F(x_{n})=e^{\beta }$. As (\ref{Ib2}) implies $d\left( x,y\right) \leq
F(x)\left( e^{v(x,y)}-1\right) \leq e^{\beta }\left( e^{v(x,y)}-1\right) $
for all $x,y\in X\setminus M$, it follows that $\left( x_{n}\right) _{n\geq
1}$ is a Cauchy sequence in $\left( X,d\right) $. Let $x\in X$ be such that $%
\underset{n\rightarrow \infty }{\lim }d\left( x_{n},x\right) =0$. Since $%
\widetilde{F}$ is continuous on $X$, $\widetilde{F}\left( x\right) =\underset%
{n\rightarrow \infty }{\lim }\widetilde{F}\left( x_{n}\right) =\underset{%
n\rightarrow \infty }{\lim }F(x_{n})\geq e^{\alpha }>0$, hence $x\in
X\setminus M$ and $\widetilde{F}\left( x\right) =F(x)$. Now $\underset{%
n\rightarrow \infty }{\lim }v\left( x_{n},x\right) =\underset{n\rightarrow
\infty }{\lim }2\log \frac{d(x_{n},x)+\max \left\{ F(x_{n}),F\left( x\right)
)\right\} }{\sqrt{F(x_{n})F(x)}}=\underset{n\rightarrow \infty }{\lim }2\log 
\frac{\max \left\{ F(x),F\left( x\right) )\right\} }{\sqrt{\left(
F(x)\right) ^{2}}}=0$, hence $\left( x_{n}\right) _{n\geq 1}$ converges to $%
x $ in $\left( X\setminus M,v\right) $.
\end{proof}

\end{document}